\documentclass[a4paper]{amsart}
\usepackage{amssymb} 
\usepackage{stmaryrd}
\usepackage{colonequals}
\usepackage[mathcal]{euscript}
\usepackage{dsfont}
\usepackage{tikz}
\usepackage{tikz-cd}
\usepackage{hyperref}
\hypersetup{%
bookmarksnumbered=true,%
colorlinks=true,%
linkcolor=blue,%
citecolor=blue,%
filecolor=blue,%
menucolor=blue,%
urlcolor=blue,%
bookmarksopen=true,%
bookmarksdepth=2,%
pageanchor=true}

\makeatletter
\@namedef{subjclassname@2020}{\textup{2020} Mathematics Subject Classification}
\makeatother

\title{The spectrum of a well-generated tensor triangulated category}

\dedicatory{To {\O}yvind Solberg on the occasion of his 60th birthday}

\usepackage[capitalise]{cleveref}

\crefformat{equation}{(#2#1#3)}
\crefrangeformat{equation}{(#3#1#4--#5#2#6)}
\crefmultiformat{equation}{(#2#1#3}{ \& #2#1#3)}{, #2#1#3}{ \& #2#1#3)}

\numberwithin{equation}{section}

\theoremstyle{plain}
\newtheorem{theorem}[equation]{Theorem}
\newtheorem{proposition}[equation]{Proposition}
\newtheorem{lemma}[equation]{Lemma} 
\newtheorem{corollary}[equation]{Corollary}
\newtheorem{thmx}{Theorem}
\theoremstyle{definition}
\newtheorem{definition}[equation]{Definition}
\newtheorem{example}[equation]{Example}

\theoremstyle{remark}
\newtheorem{remark}[equation]{Remark}

\newcommand{\cl}{\operatorname{cl}}
\newcommand{\Hom}{\operatorname{Hom}}
\newcommand{\Obj}{\operatorname{Obj}}
\newcommand{\Spec}{\operatorname{Spc}}

\newcommand{\cat}[1]{{\mathbf{#1}}}
\newcommand{\ccat}[1]{{\mathcal{#1}}}

\newcommand{\Cohcat}[1][]{\cat{Coh}_{#1}}
\newcommand{\Distcat}[1][]{\cat{DLat}_{#1}}
\newcommand{\Framecat}{\cat{Frm}}
\newcommand{\Ord}{\cat{Ord}}

\newcommand{\Topcat}{\cat{Top}}

\newcommand{\extend}[3]{{{#3}\!\!\upharpoonright_#1^#2}}
\newcommand{\restrict}[3]{{{#3}\!\!\downharpoonright^#1_#2}}
\newcommand{\op}{\mathrm{op}}
\newcommand{\compact}[2][c]{#2^{#1}}
\newcommand{\principal}[1]{\left\downarrow#1\right.}
\newcommand{\set}[2]{\left\{#1 \,\middle|\, #2\right\}}
\newcommand{\Idl}[1][]{\operatorname{Idl}_{#1}}
\newcommand{\het}[1][]{\operatorname{ht}_{#1}}
\newcommand{\pt}[1][]{\operatorname{pt}_{#1}}
\newcommand{\loc}[1][]{\operatorname{loc}^{#1}}
\newcommand{\Loc}[1][]{\operatorname{Loc}^{#1}}
\newcommand{\rad}[1][]{\operatorname{rad}^{#1}}
\newcommand{\Rad}[1][]{\operatorname{Rad}^{#1}}

\newcommand{\one}{\mathds 1}
\newcommand{\iso}{\xrightarrow{\raisebox{-.4ex}[0ex][0ex]{$\scriptstyle{\sim}$}}}
\newcommand{\logeq}{\mathrel{\ratio\Longleftrightarrow}} % logically equivalent

\newcommand{\longiso}{\xrightarrow{\ \raisebox{-.4ex}[0ex][0ex]{$\scriptstyle{\sim}$}\ }}

\newcommand{\lto}{\longrightarrow}

\newcommand{\RadT}[1]{\Rad[#1](\compact[#1]{\ccat{T}})}

\AtBeginDocument{%
 \def\MR#1{}
}
\begin{document}

\begin{abstract}
For a tensor triangulated category and any regular cardinal $\alpha$ we study the frame of $\alpha$-localizing tensor ideals and its associated space of points. For a well-generated category and its frame of localizing tensor ideals we provide conditions such that the associated space is obtained by refining the topology of the corresponding space for the triangulated subcategory of $\alpha$-compact objects. This is illustrated by several known examples for $\alpha=\aleph_0$, and new spaces arise for $\alpha>\aleph_0$.
\end{abstract}



\author[H.~Krause]{Henning Krause}
\address{Henning~Krause,
Fakult\"at f\"ur Mathematik,
Universit\"at Bielefeld,
33501 Bielefeld,
Germany}
\email{hkrause@math.uni-bielefeld.de}

\author[J.~C.~Letz]{Janina C. Letz}
\address{Janina~C.~Letz,
Fakult\"at f\"ur Mathematik,
Universit\"at Bielefeld,
33501 Bielefeld,
Germany}
\email{jletz@math.uni-bielefeld.de}

\maketitle

\setcounter{tocdepth}{1}
\tableofcontents

%%%%%%%%%%%%%%%%%%%%%%%%%%%%%%%%%%%%%%%%%%%%%%%%%
%%%%%%%%%%%%%%%%%%%%%%%%%%%%%%%%%%%%%%%%%%%%%%%%%
\section{Introduction}

For many triangulated categories arising in mathematical nature there is a concept of support, which depends on an appropriate notion of a spectrum. In order to further develop such a theory we extend basic concepts from tensor triangular geometry as initiated by Balmer \cite{Balmer:2005} to the class of well-generated triangulated categories introduced by Neeman \cite{Neeman:2001}. Our approach involves the frame of localizing tensor ideals, following the work of Kock and Pitsch \cite{Kock/Pitsch:2017}.

Given a well-generated tensor triangulated category $\ccat{T}$, we analyze the situation in which all radical localizing tensor ideals are parametrized by the open subsets of some topological space $\Spec(\ccat{T})$. Any well-generated triangulated category admits a filtration $\ccat{T}=\bigcup_{\alpha}\compact[\alpha]{\ccat{T}}$, where $\alpha$ runs through all regular cardinals and $\compact[\alpha]{\ccat{T}}$ denotes the full subcategory of $\alpha$-compact objects. The main result of this work can be summarized as follows.

\begin{thmx}\label{th:cofiltration}
The filtration $\ccat{T}=\bigcup_{\alpha}\compact[\alpha]{\ccat{T}}$ induces under appropriate assumptions a cofiltration
\begin{equation*}
\Spec(\ccat{T})=\bigcap_{\alpha}\Spec^\alpha(\compact[\alpha]{\ccat{T}})
\end{equation*}
where $\Spec^\alpha(\compact[\alpha]{\ccat{T}})$ denotes the space which parametrizes the radical $\alpha$-localizing tensor ideals of $\compact[\alpha]{\ccat{T}}$. In this cofiltration the underlying set of points remains fixed and the topology is refined when $\alpha$ increases. In particular, $\Spec(\ccat{T})=\Spec^\alpha(\compact[\alpha]{\ccat{T}})$ for some big enough regular cardinal $\alpha$.
\end{thmx} 

The theorem is motivated by various examples where a classification of localizing tensor ideals is known:
\begin{enumerate}
\item[--] the derived category of a commutative noetherian ring \cite{Neeman:1992},
\item[--] the derived category of a noetherian scheme \cite{Alonso/Jeremias/Souto:2004},
\item[--] the bootstrap category of separable $C^*$-algebras \cite{DellAmbrogio:2011},
\item[--] the stable module category of a finite group scheme \cite{Benson/Iyengar/Krause/Pevtsova:2018a},
\item[--] the category of rational $G$-spectra for any compact Lie group $G$ \cite{Greenlees:2019}.
\end{enumerate}
In each case the space $\Spec^\alpha(\compact[\alpha]{\ccat{T}})$ for $\alpha=\aleph_0$ is known, describing the thick tensor ideals of $\compact[\alpha]{\ccat{T}}$ as in Balmer's work \cite{Balmer:2005}. Then our theorem says that for any regular $\beta\geq \alpha$ the $\beta$-localizing tensor ideals of $\compact[\beta]{\ccat{T}}$ correspond to open subsets of a space which is obtained from $\Spec^\alpha(\compact[\alpha]{\ccat{T}})$ by keeping the points and refining its topology. In the above examples eventually the space is discrete.

The starting point for our work is the observation that the radical localizing tensor ideals of a tensor triangulated category form a \emph{frame}, i.e.\@ a lattice in which the infinite-join distributive law holds; see \cref{RadAlphaFrame}. There are functors relating frames and topological spaces by assigning to each frame its space of points, and to each space its frame of open subsets; cf.\@ \cite{Johnstone:1982}. These functors restrict to an equivalence called \emph{Stone duality}. The frames in this equivalence are those with \emph{enough points}. Coherent frames correspond to topological spaces that are spectral, and every coherent frame has a Hochster dual \cite{Hochster:1969}. It is important to note that not all frames have enough points. This means the notion of a frame is the more general concept, even though in many examples arising from tensor triangulated categories we can pass to its associated space of points without loss of information.

For example, when $\cat{D}(A)$ is the derived category of a commutative ring $A$, then the frame of thick tensor ideals of $\cat{D}(A)^{\aleph_0}$ (i.e.\@ the category of perfect complexes in $\cat{D}(A)$) is well understood \cite{Hopkins:1987,Neeman:1992,Thomason:1997}; in particular the frame is coherent and has therefore enough points. The space corresponding to its Hochster dual frame identifies canonically with the Zariski spectrum $\operatorname{Spec}(A)$ consisting of the set of prime ideals of $A$. There is no analogue of Hochster duality for $\compact[\alpha]{\cat{D}(A)}$ when $\alpha>\aleph_0$. For that reason we disregard the Zariski topology in this work. This also explains why localizing tensor ideals identify with open sets, and not with Thomason sets.

For any regular cardinal $\alpha$ we consider a tensor triangulated category $\ccat{K}$ with $\alpha$-coproducts, for example the category $\compact[\alpha]{\ccat{T}}$ of $\alpha$-compact objects of a well-generated tensor triangulated category $\ccat{T}$. In order to describe the radical $\alpha$-localizing tensor ideals of $\ccat{K}$ we introduce the notion of an $\alpha$-coherent frame. This is very much in line with the pioneering work of Gabriel and Ulmer \cite{Gabriel/Ulmer:1971} on locally presentable categories.\footnote{While the concept of transfinite cardinals goes back to Cantor \cite{Cantor:1895, Cantor:1897, Dauben:1979}, it was Gabriel \cite{Gabriel:1962} who noticed that any reasonable cocomplete category admits a canonical filtration $\ccat A=\bigcup_{\alpha}\ccat A^\alpha$ where $\alpha$ runs through all regular cardinals. In fact, the theory of well-generated triangulated categories builds on \cite{Gabriel/Ulmer:1971}.} A consequence is a universal support theory for objects of $\ccat{K}$, extending the work in
\cite{Kock/Pitsch:2017} for $\alpha=\aleph_0$.

We denote by $\Rad[\alpha](\ccat{K})$ the collection of radical $\alpha$-localizing tensor ideals of $\ccat{K}$ and for an object $X\in\ccat{K}$ we write $\rad[\alpha](X)$ for the smallest element in $\Rad[\alpha](\ccat{K})$ containing $X$.

\begin{thmx}[\cref{th:Rad-support}]
Let $\alpha$ be a regular cardinal and $\ccat{K}$ a tensor triangulated category with $\alpha$-coproducts. Suppose that $\Rad[\alpha](\ccat{K})$ is a set and that
\begin{equation*}
\rad[\alpha](X\otimes Y) = \rad[\alpha](X) \cap \rad[\alpha](Y)
\end{equation*}
for all $X,Y \in \ccat{K}$. Then the frame $\Rad[\alpha](\ccat{K})$ is $\alpha$-coherent and the map
\begin{equation*}
\Obj(\ccat{K})\lto \Rad[\alpha](\ccat{K})\,,\quad X\mapsto \rad[\alpha](X)
\end{equation*}
is an $\alpha$-support; it is initial among all $\alpha$-supports on $\ccat{K}$. 
\end{thmx}

There is also a version of this theorem for radical localizing tensor ideals; see \cref{th:Rad-support-infinite}. When $\Rad[\alpha](\ccat{K})$ has enough points, then this theorem can rephrased as follows. The radical $\alpha$-localizing ideals of $\ccat{K}$ are classified by the open sets of a topological space whose topology is completely determined by the support of the objects of $\ccat{K}$. 

This brings us back to the study of a well-generated tensor triangulated category $\ccat{T}$. In general there is not much hope to describe all localizing tensor ideals in terms of some topological space. This is the reason to introduce a class of well-generated tensor triangulated categories, where such a topological space is a refinement of the topological space describing the $\alpha$-localizing tensor ideals of $\compact[\alpha]{\ccat{T}}$ for some regular cardinal $\alpha$. We call such categories \emph{$\alpha$-compactly stratified}; see \cref{dfn:compactly-stratified}. Roughly speaking, this means $\ccat{T}$ is generated by $\alpha$-compact objects, a topological space classifying the radical localizing tensor ideals exists, and the canonical map $\Spec(\ccat{T})\to \Spec^\alpha(\compact[\alpha]{\ccat{T}})$ is a bijection.

For compactly generated triangulated categories (i.e.\@ $\alpha=\aleph_0$) the notion of stratification goes back to \cite{Benson/Iyengar/Krause:2011} and has been further developed in the tensor triangulated context, for example in \cite{Balmer/Favi:2011, Barthel/Heard/Sanders:2021}. Recent work of Balchin and Stevenson \cite{Balchin/Stevenson:2021} involves new topological spaces beyond that given by the category of compact objects.

With a reasonable notion of $\alpha$-compact stratification, our main theorem is based on the following.

\begin{thmx}[\cref{bijectionPointsRad}]
Let $\ccat{T}$ be a well-generated tensor triangulated category and let $\alpha\le\beta$ be regular cardinals. If $\ccat{T}$ is $\alpha$-compactly stratified, then $\ccat{T}$ is $\beta$-compactly stratified.
\end{thmx} 

While the above examples are $\aleph_0$-compactly stratified, there are good reasons to keep our approach more flexible and to allow cardinals $\alpha>\aleph_0$. In fact, well-generated triangulated categories are stable under taking localizing subcategories that are generated by sets of objects and their corresponding quotients. The following theorem illustrates the advantage of working with well-generated categories, since any set of objects of a well-generated category is $\alpha$-compact for some sufficiently big regular cardinal $\alpha$.

\begin{thmx}[\cref{co:bijectionPointsRad}]
Let $\ccat{T}$ be a well-generated tensor triangulated category and suppose that $\ccat{T}$ is $\alpha$-compactly stratified. If $\ccat{S}\subseteq\ccat{T}$ is a radical localizing tensor ideal, then the quotient $\ccat{T}/\ccat{S}$ is $\beta$-compactly stratified for some regular cardinal $\beta\ge\alpha$.
\end{thmx}

\subsubsection*{Acknowledgements}

We are grateful to several colleagues for their interest and many
helpful comments. These include Paul Balmer, Tobias Barthel, Amnon
Neeman, Greg Stevenson, and last but not least an anonymous referee.

%%%%%%%%%%%%%%%%%%%%%%%%%%%%%%%%%%%%%%%%%%%%%%%%%
%%%%%%%%%%%%%%%%%%%%%%%%%%%%%%%%%%%%%%%%%%%%%%%%%
\section{Stone duality}

We recall the definition of a frame and review the correspondence between frames and topological spaces, which is known as Stone duality.

\subsection*{Frames}

A cardinal $\alpha$ is called \emph{regular} if $\alpha$ is not the sum of fewer than $\alpha$ cardinals, all smaller than $\alpha$. For example, the cardinal $\aleph_0$ is regular because the sum of finitely many finite cardinals is finite. Also, the successor $\kappa^+$ of every infinite cardinal $\kappa$ is regular. In particular, there are arbitrarily large regular cardinals.

For a set $A$ we denote by $|A|$ its cardinality.

Let $L = (L,\le)$ be a partially ordered set (poset for short). Then $L$ is a \emph{lattice} provided all finite non-empty joins and finite non-empty meets exist. Join and meet are denoted by $\vee$ and $\wedge$, respectively. For a subset $A \subseteq L$ we write $\bigvee A$ and $\bigwedge A$ for the join and meet of the elements of $A$, respectively. If $|A| < \alpha$, we say $\bigvee A$ is an \emph{$\alpha$-join}. We denote the greatest and least elements, if they exist, by $1$ and $0$, respectively.

Fix a cardinal $\alpha$. A lattice $L$ satisfies the \emph{$\alpha$-join distributive law}, when
\begin{equation} \label{eq:alphaDL}
b \wedge \bigvee A = \bigvee_{a \in A} (b \wedge a) \quad\text{for all } b \in L \text{ and } A \subseteq L \text{ with } |A| < \alpha\,.
\end{equation}
A lattice $L$ satisfies the \emph{infinite-join distributive law} when \cref{eq:alphaDL} holds for any regular cardinal $\alpha$.

We say a lattice $L$ is \emph{$\alpha$-distributive}, if the greatest and least element exist, any $\alpha$-join exists and the $\alpha$-join distributive law holds. We define a morphism of $\alpha$-distributive lattices as a map of lattices that respects the greatest and least element, $\alpha$-joins and finite meets. The category of $\alpha$-distributive lattices is denoted by $\Distcat[\alpha]$.

It is common to drop the cardinal $\alpha$ when $\alpha=\aleph_0$. So a lattice is \emph{distributive} if it is $\aleph_0$-distributive, and $\Distcat=\Distcat[\aleph_0]$. 

A lattice $F$ is called a \emph{frame} if it is $\alpha$-distributive for every cardinal $\alpha$. We define a morphism of frames as a map of lattices that respects arbitrary joins and finite meets. The category of frames is denoted by $\Framecat$.

\subsection*{Prime elements versus points}

An element $p$ of a lattice $L$ is \emph{prime}, if $p \neq 1$ and
\begin{equation*}
a \wedge b \leq p \,\implies\, a \leq p \text{ or } b \leq p\,.
\end{equation*}
In a frame $F$ the prime elements correspond to the points. A \emph{point} of a frame $F$ is a frame morphism $x\colon F\to\{0,1\}$. Given a point $x$, the corresponding prime element is 
\begin{equation*}
p_x \colonequals \bigvee x^{-1}(0)\,.
\end{equation*}
The set of points $\pt(F)$, and thus the set of prime elements, form a topological space with the open sets given by
\begin{equation*}
U(a) \colonequals \set{x \in \pt(F)}{x(a) = 1} \quad\text{for } a \in F\,.
\end{equation*}

\subsection*{Frames versus spaces}

We denote the category of topological spaces with continuous maps by $\Topcat$. For a topological space we consider the frame of its open sets; this is ordered by inclusion, joins are given by unions, and finite meets by intersections. Sending a topological space to its frame of open sets yields a contravariant functor $\Omega \colon \Topcat \to \Framecat^\op$. This functor has a right adjoint, the functor of points $\pt \colon \Framecat^\op \to \Topcat$. 

The topological spaces occurring as the space of points of a frame are precisely the sober spaces; these include any Hausdorff space and any topological space underlying a scheme. Recall that a topological space is \emph{sober}, if every irreducible closed set has a unique generic point; that is if $V$ is an irreducible closed set, then there exists a unique $x \in V$ with $\cl(x) = V$.

The adjunction between topological spaces and frames
\begin{equation} \label{eq:AdjointOmegaPt}
\begin{tikzcd}[column sep=huge]
\Topcat \ar[r,shift left,"\Omega"] & \Framecat^\op \ar[l,shift
left,"\pt"]
\end{tikzcd}
\end{equation}
identifies the sober spaces with the spatial frames \cite[II.1.7]{Johnstone:1982}; this is called \emph{Stone duality}. A frame $F$ is \emph{spatial} if it has enough points: for all $a,b\in F$ satisfying $a\not\leq b$ there is a point $x$ such that $x(a)=1$ and $x(b)=0$.

%%%%%%%%%%%%%%%%%%%%%%%%%%%%%%%%%%%%%%%%%%%%%%%%%
%%%%%%%%%%%%%%%%%%%%%%%%%%%%%%%%%%%%%%%%%%%%%%%%%
\section{Coherent frames and Hochster duality}

Any frame may be viewed as a category, and following the theory of
locally $\alpha$-presentable categories \cite{Gabriel/Ulmer:1971} we
develop the theory of $\alpha$-coherent frames. This generalizes the
well-known notion of a coherent frame. More specifically, one may think of a lattice $L$ as a category $\ccat L$ where the objects of $\ccat L$ are the elements of $L$ and
\begin{equation*}
\Hom(a,b) \neq \varnothing \,\iff\, |\Hom(a,b)| = 1 \,\iff\, a \leq b \quad\text{for } a, b \in L\,.
\end{equation*}
Note that meet and join in $L$ correspond to product and coproduct in
$\ccat L$, respectively. Via the assignment $L\mapsto\ccat L$ an
$\alpha$-coherent frame corresponds to a locally $\alpha$-presentable category in the sense of \cite{Gabriel/Ulmer:1971}.

\subsection*{Coherent frames}

Let $L$ be a lattice and $\alpha$ a cardinal. An element $a \in L$ is \emph{$\alpha$-compact}, if for every set $A \subseteq L$ whose join exists and $a \leq \bigvee A$ there exists a subset $B \subseteq A$ with $|B| < \alpha$ and $a \leq \bigvee B$. We denote the subset of $\alpha$-compact elements of $L$ by $\compact[\alpha]{L}$.

A frame $F$ is \emph{$\alpha$-coherent}, if
\begin{enumerate}
\item[(C1)] every element is the join of $\alpha$-compact elements,
\item[(C2)] the greatest element $1$ is $\alpha$-compact, and
\item[(C3)] the $\alpha$-compact elements are closed under $\alpha$-joins and finite meets.
\end{enumerate}
The latter two conditions imply that $\compact[\alpha]{F}$ is an $\alpha$-distributive sublattice of $F$. When $\alpha$ is regular, then $\compact[\alpha]{F}$ is closed under $\alpha$-joins. We define a morphism of $\alpha$-coherent frames as a morphism of frames that restricts to a morphism on the sublattices of $\alpha$-compact elements. We denote the category of $\alpha$-coherent frames by $\Cohcat[\alpha]$.

For cardinals $\alpha \leq \beta$, we have $\compact[\alpha]{F}\subseteq \compact[\beta]{F}$. Moreover, we have $F = \compact[\alpha]{F}$ for $\alpha=\sup(|F|^+,\aleph_0)$, and therefore every frame is $\alpha$-coherent for some regular cardinal $\alpha$.

It is common to drop the cardinal $\alpha$ when $\alpha=\aleph_0$. So an element is \emph{compact} if it is $\aleph_0$-compact, a frame is \emph{coherent} if it is $\aleph_0$-coherent, and $\Cohcat=\Cohcat[\aleph_0]$. Any coherent frame is spatial \cite[II.3.4]{Johnstone:1982}. However, the same does not hold for $\alpha$-coherent frames when $\alpha>\aleph_0$, since any frame is $\alpha$-coherent for some sufficiently big cardinal $\alpha$.

\begin{lemma} \label{coherentAscent}
Let $\alpha \leq \beta$ be regular cardinals. 
\begin{enumerate}
\item\label{coherentAscent:Ascent} If $F$ is an $\alpha$-coherent frame, then the $\beta$-compact elements are precisely the $\beta$-joins of $\alpha$-compact elements, and $F$ is $\beta$-coherent.
\item\label{coherentAscent:Compact} If $F$ is a $\beta$-coherent frame, then
\begin{equation*}
\compact[\alpha]{(\compact[\beta]{F})} = \compact[\alpha]{F}\,.
\end{equation*}
\end{enumerate}
\end{lemma}
\begin{proof}
We first assume $F$ is $\alpha$-coherent. We need to show that any $\beta$-join of $\alpha$-compact elements is $\beta$-compact. Let $A \subseteq \compact[\alpha]{F}$ with $|A| < \beta$ and $a \colonequals \bigvee A$. We assume
\begin{equation*}
a \leq \bigvee B \quad\text{for } B \subseteq F\,.
\end{equation*}
Then $a' \leq \bigvee B$ for all $a' \in A$. Since any $a'$ is $\alpha$-compact, there exists $B_{a'} \subseteq B$ with $|B_{a'}| < \alpha$ and $a' \leq \bigvee B_{a'}$. Then
\begin{equation*}
a \leq \bigvee A \leq \bigvee_{a' \in A} \bigvee B_{a'}\,.
\end{equation*}
Since $\sum_{a' \in A} |B_{a'}| < \beta$, this shows that $a$ is $\beta$-compact. 

To show $F$ is $\beta$-coherent, it is enough to show that the meet $a \wedge b$ of $\beta$-compact elements $a$ and $b$ is $\beta$-compact. By the first assertion we can write
\begin{equation*}
a = \bigvee A \quad\text{and}\quad b = \bigvee B \quad\text{for } A,B \subseteq \compact[\alpha]{F}
\end{equation*}
with $|A|, |B| < \beta$. Then
\begin{equation*}
a \wedge b = \left(\bigvee A\right) \wedge \left(\bigvee B\right) = \bigvee_{a' \in A} \bigvee_{b' \in B} (a' \wedge b')\,.
\end{equation*}
Since $\beta$ is regular, this is a $\beta$-join. In particular $a \wedge b$ is $\beta$-compact. 

It remains to show $\compact[\alpha]{(\compact[\beta]{F})} = \compact[\alpha]{F}$ when $F$ is $\beta$-coherent. It is clear that any $\alpha$-compact element of $F$ is $\alpha$-compact in $\compact[\beta]{F}$. For the opposite inclusion, we take $a \in \compact[\alpha]{(\compact[\beta]{F})}$. Let $a \leq \bigvee B$ for some $B \subseteq F$. Since $F$ is $\beta$-coherent, any element $b \in B$ can be expressed as the join of $\beta$-compact elements, say of $C_b \subseteq \compact[\beta]{F}$. Then
\begin{equation*}
a \leq \bigvee_{b \in B} \bigvee C_b\,.
\end{equation*}
This is a join of $\beta$-compact elements. Since $a$ is $\alpha$-compact in $\compact[\beta]{F}$, there exists $C \subseteq \bigcup_{b \in B} C_b$ with $|C| < \alpha$ and
\begin{equation*}
a \leq \bigvee C \leq \bigvee B'\,,
\end{equation*}
where $B' \subseteq B$ is chosen such that for every $c \in C$ there exists $b \in B'$ with $c \in C_b$. This subset $B'$ can be chosen such that $|B'| < \alpha$.
\end{proof}

\subsection*{Frame of ideals}

Let $\alpha$ be a regular cardinal. An \emph{$\alpha$-ideal} in an $\alpha$-distributive lattice $L$ is a non-empty subset $I$ of $L$ that is closed under $\alpha$-joins, and
\begin{equation*}
a \in I \quad\text{and}\quad b \leq a \,\implies\, b \in I\,.
\end{equation*}

The set of $\alpha$-ideals of $L$ form a lattice, ordered by inclusion; we denote it by $\Idl[\alpha](L)$. The meet of a set of $\alpha$-ideals is their intersection. The join of a set $C \subseteq \Idl[\alpha](L)$ is the smallest $\alpha$-ideal containing their union and is explicitly given by
\begin{equation*}
\bigvee C = \set{a \in L}{a = \bigvee A \text{ for some }A \subseteq \bigcup_{J \in C} J \text{ with } |A| < \alpha}\,.
\end{equation*}
This is an $\alpha$-ideal, since $\alpha$-distributivity of $L$ implies that any $b \leq \bigvee A$ lies also in the right-hand side. 

As before we drop the cardinal when $\alpha=\aleph_0$. Thus an $\aleph_0$-ideal is called \emph{ideal}. 

The \emph{principal ideal} of $a \in L$ is
\begin{equation*}
\principal{a} \colonequals \set{b \in L}{b \leq a}\,.
\end{equation*}
A principal ideal is an $\alpha$-ideal for any cardinal $\alpha$. 

\begin{lemma} \label{idl-infinite-dist}
For an $\alpha$-distributive lattice $L$, the lattice of $\alpha$-ideals $\Idl[\alpha](L)$ satisfies the infinite-join distributive law. In particular, the $\alpha$-ideals form a frame.
\end{lemma}
\begin{proof}
Suppose $I$ is an $\alpha$-ideal and $C \subseteq \Idl[\alpha](L)$. We always have
\begin{equation*}
\bigvee_{J \in C} (I \wedge J) \leq I \wedge \bigvee C\,.
\end{equation*}
We want to show that equality holds. We take $a \in I \wedge \bigvee C$. So $a \in \bigvee C$ and there exists a set $A \subseteq \bigcup_{J \in C} J$ with $|A| < \alpha$ such that $a = \bigvee A$. Since also $a \in I$ we have $A \subseteq I$. Now the infinite-join distributive law for sets yields
\begin{equation*}
A \subseteq I \cap \bigcup_{J \in C} J = \bigcup_{J \in C} (I \cap J)\,.
\end{equation*}
Thus $a = \bigvee A \in \bigvee_{J \in C} (I \wedge J)$. 
\end{proof}

\subsection*{Coherent frames as completions}

The following result and its proof are an extension of \cite[II.3.2]{Johnstone:1982} about coherent frames and a special case of Gabriel--Ulmer duality \cite[\S7]{Gabriel/Ulmer:1971} about locally presentable categories. In particular, we see that any $\alpha$-coherent frame $F$ arises as a completion of its sublattice $F^\alpha$.

\begin{theorem} \label{CohcatDistcatEquiv}
Let $\alpha$ be a regular cardinal. Then we have a pair of mutually quasi-inverse equivalences
\begin{equation*}
\begin{tikzcd}[column sep=huge]
\Cohcat[\alpha] \ar[r,shift left=0.5em,"{\compact[\alpha]{(-)}}"] \ar[r,phantom,"\equiv"] & \Distcat[\alpha]\,. \ar[l,shift left=0.5em,"{\Idl[\alpha]}"]
\end{tikzcd}
\end{equation*}
\end{theorem}

\begin{proof}
We begin by showing that $\compact[\alpha]{(-)}$ is well-defined. If $F$ is an $\alpha$-coherent frame, then $\compact[\alpha]{F}$ is an $\alpha$-distributive lattice by definition. For morphisms, we observe that a morphism of $\alpha$-coherent frames restricts to a morphism of the sublattices formed by the $\alpha$-compact elements. In particular, it preserves $\alpha$-joins.

Next we show that $\compact[\alpha]{(\Idl[\alpha](L))}$ equals the set of principal ideals of $L$. Let $L \in \Distcat[\alpha]$. Choose $a \in L$ and $C \subseteq \Idl[\alpha](L)$. Then
\begin{align*}
\principal{a} \leq \bigvee C &\iff{} a \in \bigvee C \\
&\iff{} a = \bigvee A \quad\text{for some } A \subseteq \bigcup_{I \in C} I \text{ with } |A| < \alpha \\
& \iff{} a \leq \bigvee C' \quad\text{for some } C' \subseteq C \text{ with } |C'| < \alpha\,.
\end{align*}
It follows that any principal ideal of $L$ is $\alpha$-compact in $\Idl[\alpha](L)$. 

For the reverse inclusion, take $I \in \compact[\alpha]{(\Idl[\alpha](L))}$. We can express $I$ as the join of principal ideals of its elements. Since $I$ is $\alpha$-compact there is a subset $A \subseteq I$ with $|A| < \alpha$, such that
\begin{equation*}
I = \bigvee_{a \in I} \principal{a} = \bigvee_{a \in A} \principal{a} = \principal{\left(\bigvee A\right)}\,.
\end{equation*}
The last equality holds, since $L$ is closed under $\alpha$-joins. 

We show that $\Idl[\alpha]$ is well-defined. By \cref{idl-infinite-dist} the lattice $\Idl[\alpha](L)$ is a frame. Since the $\alpha$-compact elements of $\Idl[\alpha](L)$ are precisely the principal ideals, they form a sublattice, which is closed under $\alpha$-joins, and every ideal can be expressed as a join of principal ideals.

Finally we show that $F \cong \Idl[\alpha](\compact[\alpha]{F})$. In fact, we show that the canonical map
\begin{equation*}
\varphi \colon F \to \Idl[\alpha](\compact[\alpha]{F}) \,,\quad a \mapsto \set{b \in \compact[\alpha]{F}}{b \leq a}
\end{equation*}
is an isomorphism of $\alpha$-coherent frames. Since $F$ is $\alpha$-coherent, every element can be expressed as a join of $\alpha$-compact elements. In particular, one has
\begin{equation*}
\bigvee \varphi(a) = a \quad\text{for all } a \in F\,.
\end{equation*}
Thus $\varphi$ is injective. If $I \in \Idl[\alpha](\compact[\alpha]{F})$, then
\begin{equation*}
\varphi\left(\bigvee I\right) = I\,,
\end{equation*}
and $\varphi$ is surjective. The map $\varphi$ respects the orderings of $F$ and $\Idl[\alpha](\compact[\alpha]{F})$. Finally, we observe that $\varphi$ maps $\alpha$-compact elements to $\alpha$-compact elements, since the $\alpha$-compact elements of $\Idl[\alpha](\compact[\alpha]{F})$ are precisely the principal ideals of $\compact[\alpha]{F}$.
\end{proof}

\begin{remark}
For any regular cardinal $\alpha$, let $\cat{Lat}_\alpha$ denote the category of complete lattices such that every element is a join of $\alpha$-compact elements, and let $\cat{Pos}_\alpha$ denote the category of posets which admit $\alpha$-joins. Then we have a pair of mutually quasi-inverse equivalences
\begin{equation*}
\begin{tikzcd}[column sep=huge]
\cat{Lat}_\alpha \ar[r,shift left=0.5em,"{\compact[\alpha]{(-)}}"] \ar[r,phantom,"\equiv"] & \cat{Pos}_\alpha \ar[l,shift left=0.5em,"{\Idl[\alpha]}"]
\end{tikzcd}
\end{equation*}
which restricts to the one in \cref{CohcatDistcatEquiv}. This is precisely the statement of \cite[Korollar~7.11]{Gabriel/Ulmer:1971} when specialized to categories of posets.
\end{remark}

\subsection*{Change of cardinals}

For regular cardinals $\alpha \leq \beta$ we obtain from \cref{CohcatDistcatEquiv} the following (non-commutative) diagram
\begin{equation*} \label{e:CoherentDistCat}
\begin{tikzcd}[column sep=huge]
\Cohcat[\beta] \ar[r,shift left=0.5em,"{\compact[\beta]{(-)}}"] \ar[r,phantom,"\equiv"] & \Distcat[\beta] \ar[l,shift left=0.5em,"{\Idl[\beta]}"] \ar[d,tail] \\
\Cohcat[\alpha] \ar[r,shift left=0.5em,"{\compact[\alpha]{(-)}}"] \ar[r,phantom,"\equiv"] \ar[u,tail] & \Distcat[\alpha] \ar[l,shift left=0.5em,"{\Idl[\alpha]}"] 
\end{tikzcd}
\end{equation*}
where the arrows $\rightarrowtail$ indicate forgetful functors. The vertical arrow on the left is well-defined; by \cref{coherentAscent} any $\alpha$-coherent frame is $\beta$-coherent and any morphism of $\alpha$-coherent frames is a morphism of $\beta$-coherent frames. We observe that
\begin{equation*}
\bigcup_{\alpha\text{ regular}}\Cohcat[\alpha]=\Framecat\qquad\text{and}\qquad
\bigcap_{\alpha\text{ regular}}\Distcat[\alpha]=\Framecat
\end{equation*}
and then both horizontal functors identify with the identity.

The following lemma explains the commutation rules in the above diagram.

\begin{lemma}
The composite of functors 
\begin{equation*} \label{eq:BetaCompletionDistAlpha}
\begin{tikzcd}
\Distcat[\alpha] \ar[r,"{\Idl[\alpha]}"] & \Cohcat[\alpha] \ar[r,tail] & \Cohcat[\beta] \ar[r,"{\compact[\beta]{(-)}}"] & \Distcat[\beta]
\end{tikzcd}
\end{equation*}
is left adjoint to the forgetful functor $\Distcat[\beta] \rightarrowtail \Distcat[\alpha]$. Similarly, the functor
\begin{equation*} \label{eq:CompletionDistAlpha}
\begin{tikzcd}
\Distcat[\alpha] \ar[r,"{\Idl[\alpha]}"] & \Cohcat[\alpha] \ar[r,tail] & \Framecat
\end{tikzcd}
\end{equation*}
is left adjoint to the forgetful functor $\Framecat \rightarrowtail \Distcat[\alpha]$. 
\end{lemma}
\begin{proof}
Let $L$ be an $\alpha$-distributive lattice and $M$ a $\beta$-distributive lattice. It is straightforward to check that the morphisms
\begin{equation*}
\begin{tikzcd}[row sep=0]
L \ar[r] & \compact[\beta]{(\Idl[\alpha](L))} \\
a \ar[r,mapsto] & \principal{a}
\end{tikzcd}
\quad\text{and}\quad
\begin{tikzcd}[row sep=0]
\compact[\beta]{(\Idl[\alpha](M))} \ar[r] & M \\
I \ar[r,mapsto] & \bigvee I
\end{tikzcd}
\end{equation*}
are unit and counit, respectively. We note that the second morphism is well-defined, since for $I \in \compact[\beta]{(\Idl[\alpha](M))}$ we have
\begin{equation*}
I = \bigvee_{a \in I} \principal{a} = \bigvee_{a \in A} \principal{a}
\end{equation*}
for some $A \subseteq I$ with $|A| < \beta$. Then $\bigvee I = \bigvee A$ exists in $M$. 

The unit and counit morphisms for the second claim are defined similarly.
\end{proof}

\subsection*{Spectral spaces}

Let $X$ be a topological space and $\alpha$ a cardinal. We say an open set $U$ is \emph{$\alpha$-compact}, if for any open cover $U\subseteq\bigcup_{U'\in \mathcal{U}} U'$ there exists $\mathcal{U}' \subseteq \mathcal{U}$ with $|\mathcal{U}'|<\alpha$ and $U\subseteq\bigcup_{U' \in \mathcal{U}'} U'$. When $\alpha = \aleph_0$, then this notion coincides with classical quasi-compactness.

A topological space is \emph{$\alpha$-spectral}, if it is sober and the $\alpha$-compact open sets form an $\alpha$-distributive sublattice that is a basis for the topology; that means:
\begin{enumerate}
\item[(SP1)] any open set is a union of $\alpha$-compact open sets,
\item[(SP2)] the space is $\alpha$-compact, and
\item[(SP3)] any $\alpha$-union and any finite intersection of $\alpha$-compact open sets is $\alpha$-compact.
\end{enumerate}
When $\alpha$ is a regular cardinal, then the closure under $\alpha$-unions always holds. 

We define a morphism between $\alpha$-spectral spaces as a continuous map for which the inverse image of an $\alpha$-compact open set is $\alpha$-compact; cf.\@ \cite{Hochster:1969}.

It is common to drop the cardinal $\alpha=\aleph_0$, calling a space \emph{spectral} if it is $\aleph_0$-spectral. 

\begin{proposition}
Let $\alpha$ be a regular cardinal. The adjunction between topological spaces and frames \eqref{eq:AdjointOmegaPt} restricts to a contravariant equivalence between the category of $\alpha$-spectral spaces and the category of spatial $\alpha$-coherent frames.
\end{proposition}
\begin{proof}
Let $X$ be a sober space and $\Omega(X)$ its corresponding spatial frame. It is obvious that the $\alpha$-compact open sets of $X$ identify with the $\alpha$-compact elements of $\Omega(X)$. Thus the claim holds. 
\end{proof}

\subsection*{Prime ideals}

Let $L$ be a lattice. An ideal $I\subseteq L$ is \emph{prime} if $1\not\in I$ and
\begin{equation*}
a \wedge b \in I \,\implies\, a \in I\text{ or } b \in I\,.
\end{equation*}
If $L$ is a distributive lattice, then the prime ideals are precisely the prime elements of $\Idl[\aleph_0](L)$. 

\subsection*{Hochster duality}

The correspondence between topological spaces and frames \eqref{eq:AdjointOmegaPt} identifies spectral spaces with coherent frames. Following Hochster \cite{Hochster:1969} we assign to a coherent frame $F$ its \emph{Hochster dual} frame given by
\begin{equation*}
F^\vee \colonequals \Idl[\aleph_0]((\compact[\aleph_0]{F})^\op)\,.
\end{equation*}
It follows from \cref{CohcatDistcatEquiv} that $F^\vee$ is coherent and $F^{\vee\vee}\cong F$. 

For a coherent frame $F$ any point is determined by its values on $\compact[\aleph_0]{F}$. That is, the assignment $x\mapsto x^{-1}(0)\cap \compact[\aleph_0]{F}$ gives a bijection between the set of points of $F$ and the set of prime ideals in $\compact[\aleph_0]{F}$. Furthermore, the complements of the prime ideals of $\compact[\aleph_0]{F}$ are precisely the prime ideals of $(\compact[\aleph_0]{F})^\op \cong \compact[\aleph_0]{(F^\vee)}$. Thus there is a canonical bijection of sets $\pt(F) \iso \pt(F^\vee)$.

This bijection identifies Thomason subsets of $\pt(F)$ with the open subsets of $\pt(F^\vee)$. A set is \emph{Thomason} if it of the form $\bigcup_i V_i$ where each $V_i$ is closed with quasi-compact complement.

%%%%%%%%%%%%%%%%%%%%%%%%%%%%%%%%%%%%%%%%%%%%%%%%%
%%%%%%%%%%%%%%%%%%%%%%%%%%%%%%%%%%%%%%%%%%%%%%%%%
\section{Refining the topology on the space of points}

In this section we consider for any frame its space of points and study the question when an injective morphisms of frames corresponds to a refinement of the topology. We use the specialization order on the set of points and need to assume that points are locally closed.

\subsection*{The specialization order}

For a frame $F$ we consider on the set of points $\pt(F)$ the following partial order:
\begin{align*}
x \leq y \,&\logeq \,\, x(a) \geq y(a)\text{ for all } a\in F\\
&\iff \, p_x \leq p_y \\
&\iff \, \cl(x) \supseteq \cl(y)\,,
\end{align*}
where $\cl(x)$ denotes the closure of the point $x$. In particular,
\begin{equation*}
\cl(x) = \set{y\in\pt(F)}{x\le y}\,.
\end{equation*}
When $\varphi\colon F\to G$ is a morphism of frames, then $\pt(\varphi)$ preserves the order. For $x\in\pt(G)$ set $y \colonequals \pt(\varphi)(x) = x\circ \varphi$. Then we have
\begin{equation*}
p_y = \bigvee (x \circ \varphi)^{-1}(0) = \bigvee \set{a \in F}{\varphi(a) \leq p_x}
\end{equation*}
and therefore $\varphi(p_y) \le p_x$. We say $x$ is \emph{$\varphi$-honest} if this is an equality.

\subsection*{Locally closed points}

A point $x$ of a space $X$ is called \emph{locally closed}, if there exists an open set $U$ and a closed set $V$ of $X$ such that $U \cap V = \{x\}$.\footnote{If $X$ is spectral, this means that $x$ is \emph{trop-beau} \cite{Balmer/Favi:2011}, \emph{visible} \cite{Stevenson:2014}, and \emph{weakly visible} \cite{Barthel/Heard/Sanders:2021} in the corresponding Hochster dual space.} By taking $V=\cl(x)$ we see that a point $x$ is locally closed if and only if there exists an open set $U$ of $X$ such that 
\begin{equation*}
U \cap \set{y\in X}{y\ge x} = \{x\}\,.
\end{equation*}
This shows that a point $x$ is locally closed if and only if there exists an open set $U$ in which $x$ is maximal.

\begin{lemma} \label{strongly-visible}
Let $\varphi \colon F \to G$ be a map of frames such that the induced map $\pt(\varphi)$ is injective. If every point of $F$ is locally closed, then every point of $G$ is locally closed.
\end{lemma}
\begin{proof}
Set $f\colonequals \pt(\varphi)$. Let $x\in\pt(G)$ and $U\subseteq \pt(F)$ open such that $f(x)$ is maximal in $U$. Then $x$ is maximal in $f^{-1}(U)$, since any $y>x$ is mapped to $f(y)> f(x)$. Thus $x$ is locally closed.
\end{proof}

\begin{lemma} \label{maximalEltsBijPoints}
Let $\varphi \colon F \to G$ be a map of spatial frames such that the induced map $\pt(\varphi)$ is injective. Then the pre-image of every maximal point of $F$ is $\varphi$-honest.
\end{lemma}
\begin{proof}
Let $x \in \pt(G)$ such that $y \colonequals \pt(\varphi)(x)$ is maximal. Then $x$ is also maximal.

We assume $\varphi(p_y) < p_x$. Since $G$ is spatial, there exists $x'\in\pt (G)$ such that $\varphi(p_y) \leq p_{x'}\neq p_x$. Then for $y' \colonequals \pt(\varphi)(x')$ we have
\begin{equation*}
p_y\leq \bigvee \set{a \in F}{\varphi(a) \leq p_{x'}}=p_{y'}
\end{equation*}
and this is an equality since $y$ is maximal. This is a contradiction to the assumption that $\pt(\varphi)$ is injective, since $x\neq x'$. Thus we obtain $\varphi(p_y) = p_x$ and $x$ is $\varphi$-honest.
\end{proof}

\begin{proposition} \label{BijPoints} 
Let $\varphi \colon F \to G$ and $\psi \colon G\to H$ be injective maps of spatial frames. Suppose that every point of $F$ is locally closed. Then the maps induced on the points $\pt(\varphi)$ and $\pt(\psi)$ are bijective if and only if $\pt(\psi\circ\varphi)$ is bijective.
\end{proposition}
\begin{proof}
One direction is clear. Thus assume that $\pt(\psi\circ\varphi)$ is a bijection. We wish to show that $\pt(\varphi)$ is bijective. We pick $z\in\pt(F)$ and need to show that $\pt(\varphi)^{-1}(z)$ contains only one element. Let $U$ be an open subset of $\pt(F)$ such that 
\begin{equation*}
U \cap \set{y\in\pt(F)}{y\ge z}=\{z\}\,.
\end{equation*}
Set $V\colonequals\pt(\varphi)^{-1}(U)$ and $W\colonequals\pt(\psi)^{-1}(V)$. We consider the frame
\begin{equation*}
F' \colonequals \set{a\in F}{U(a)\subseteq U}
\end{equation*}
with $\pt(F')=U$. Analogously one defines
\begin{equation*}
G' \colonequals \set{a\in G}{U(a)\subseteq V} \quad\text{and}\quad H' \colonequals \set{a\in H}{U(a)\subseteq W}\,.
\end{equation*}
Then $\varphi$ and $\psi$ restrict to injective maps of spatial frames $F' \to G'$ and $G'\to H'$. Note that $z$ is a maximal point of $F'$. By assumption $\pt(\psi \circ \varphi)$ is bijective and we can pick $x\in\pt(H')$ and $y\in\pt(G')$ such that $\pt(\psi)(x)=y$ and $\pt(\varphi)(y)=z$. We apply \cref{maximalEltsBijPoints} and have $(\psi\circ\varphi)(p_z)=p_x$. On the other hand, $\varphi(p_z)\leq p_y$, and this is also an equality, since $\psi$ is injective and
\begin{equation*}
\psi(\varphi(p_z))\le \psi(p_y)\le p_x\,.
\end{equation*}
Let $y'\in\pt(G')$ with $\pt(\varphi)(y')=z$. Then we have $p_y= \varphi(p_z)\leq p_{y'}$, and therefore $p_x= \psi(p_y)\le\psi(p_{y'})$. This implies $y'=y$ since $x$ is maximal and $\psi$ injective. Thus $\pt(\varphi)^{-1}(z)$ contains only one element, and we conclude that $\pt(\varphi)$ is bijective. Clearly, then also $\pt(\psi)$ is bijective.
\end{proof}

\subsection*{Chain conditions}

% We provide a useful criterion such that all points of a frame are locally closed. Let $F$ be an $\alpha$-coherent frame. If $(\compact[\alpha]{F},\leq)$ satisfies the descending (ascending) chain condition, then the meet (join) of any subset $A$ of $\compact[\alpha]{F}$ exists and coincides with the meet (join) of a finite subset of $A$. Moreover, the $\alpha$-meet ($\alpha$-join) distributive law holds. Thus we see that $(\compact[\alpha]{F},\leq)$ satisfies the ascending chain condition if and only if $(F,\leq)$ does. The same need not hold for the descending chain condition; while arbitrary joins exist in $\compact[\alpha]{F}$, they need not coincide with those in $F$, since the join distributive law need not hold in $\compact[\alpha]{F}$.
We provide a useful criterion such that all points of a frame are locally closed. Let $F$ be an $\alpha$-coherent frame. If $(\compact[\alpha]{F},\leq)$ satisfies the descending (ascending) chain condition, then the meet (join) of any subset $A$ of $\compact[\alpha]{F}$ exists and coincides with the meet (join) of a finite subset of $A$. Thus we see that $(\compact[\alpha]{F},\leq)$ satisfies the ascending chain condition if and only if $(F,\leq)$ does. The same need not hold for the descending chain condition; since arbitrary meets exist in $\compact[\alpha]{F}$ so do arbitrary joins, but the latter need not coincide with those in $F$.

\begin{lemma} \label{pointsClosedDownwardsOpen}
Let $F$ be an $\alpha$-coherent frame such that $\compact[\alpha]{F}$ satisfies the descending chain condition. Then for any point $x \in \pt(F)$ the set
\begin{equation*}
\set{y \in \pt(F)}{x \geq y}
\end{equation*}
is open. 
\end{lemma}
\begin{proof}
Let $A \colonequals x^{-1}(1) \cap \compact[\alpha]{F}$ and $a \colonequals \bigwedge A$. Since $\compact[\alpha]{F}$ satisfies the descending chain condition, any meet of elements in $\compact[\alpha]{F}$ is finite; in particular $a \in \compact[\alpha]{F}$. We show now that $U(a) = \set{y \in \pt(F)}{x \geq y}$. 

For $y \leq x$ we have $y(a) \geq x(a) = 1$. That is $y(a) = 1$ and $y \in U(a)$. 

Let $y \in U(a)$ and $b \in \compact[\alpha]{F}$ with $x(b) = 1$. Then $a \leq b$ and thus $1 = y(a) \leq y(b)$ and $y(b) = 1$. That means
\begin{equation*}
x(b) \leq y(b) \quad\text{for all } b \in \compact[\alpha]{F}\,.
\end{equation*}
Since $F$ is $\alpha$-coherent, the points are completely determined by their value on the $\alpha$-compact elements. So $x \geq y$. 
\end{proof}

\subsection*{Dimension}

For a frame $F$ the partial order on $\pt(F)$ suggests the definition of a dimension. This goes back to Gulliksen \cite{Gulliksen:1973} and will not be used in the sequel; but it is the right concept for proving results for tensor triangulated categories, as for example in \cite{Balmer:2007}.

Let $(X,\le)$ be a poset satisfying the descending chain condition. We write $\Ord$ for the class of ordinals. Each set of ordinals has a supremum and we set $\sup\varnothing=0$. The \emph{height} of $x\in X$ is by definition
\begin{equation*}
\het(x) \colonequals \sup\set{\het(y)+1}{y<x}\,.
\end{equation*}
It is clear from the definition that
\begin{enumerate}
\item if $x \leq y$, then $\het(x) \leq \het(y)$, and
\item if $x \leq y$ and $\het(x) = \het(y)$, then $x = y$.
\end{enumerate}

Given a frame $F$ such that $(\pt(F),\le)$ satisfies the descending chain condition, one defines for $a\in F$ the \emph{dimension}
\begin{equation*}
\dim(a) \colonequals \sup_{x\in U(a)}\het(x)\,.
\end{equation*}
This yields an order preserving map $F\to\Ord$. For example, this gives the usual dimension of a noetherian spectral space $X$ when applied to $F=\Omega(X)^\vee$.

%%%%%%%%%%%%%%%%%%%%%%%%%%%%%%%%%%%%%%%%%%%%%%%%%
%%%%%%%%%%%%%%%%%%%%%%%%%%%%%%%%%%%%%%%%%%%%%%%%%
\section{The frame of localizing tensor ideals}

For a tensor triangulated category we study the frame of radical localizing tensor ideals. We assume that the category admits $\alpha$-coproducts and consider $\alpha$-localizing tensor ideals. This frame yields a universal notion of support, following the approach of \cite{Balmer:2005} and \cite{Kock/Pitsch:2017}.

\subsection*{Localizing tensor ideals}

Fix a tensor triangulated category $(\ccat{K},\otimes,\one)$ and a regular cardinal $\alpha$. An \emph{$\alpha$-coproduct} in $\ccat{K}$ is a coproduct of fewer than $\alpha$ objects. We assume that $\alpha$-coproducts exist in $\ccat{K}$. By this we also mean that these coproducts are preserved by the tensor product.

A triangulated subcategory of $\ccat{K}$ is \emph{$\alpha$-localizing}, if it is closed under $\alpha$-coproducts and direct summands. The latter condition is for free when $\alpha>\aleph_0$. A triangulated subcategory $\ccat{I}\subseteq\ccat{K}$ is a \emph{tensor ideal} if $X\in\ccat{I}$ and $Y\in \ccat{K}$ implies $X\otimes Y\in\ccat{I}$. For any class of objects $\ccat X\subseteq\ccat{K}$ we denote by $\loc[\alpha](\ccat{X})$ the smallest $\alpha$-localizing tensor ideal of $\ccat{K}$ containing $\ccat{X}$. We write $\Loc[\alpha](\ccat{K})$ for the collection of $\alpha$-localizing tensor ideals of $\ccat{K}$. When $\Loc[\alpha](\ccat{K})$ is a set, then it is a lattice where the meet of a subset of $\Loc[\alpha](\ccat{K})$ is their intersection, and the join the smallest $\alpha$-localizing tensor ideal containing their union.

When $\ccat{K}$ has arbitrary coproducts, then a subcategory is \emph{localizing} if it is $\alpha$-localizing for every cardinal $\alpha$. For $\ccat X\subseteq\ccat{K}$ we set 
\begin{equation*}
\loc(\ccat{X}) \colonequals \bigcup_{\alpha\text{ regular}}\loc[\alpha](\ccat{X})
\end{equation*}
and write $\Loc(\ccat{K})$ for the collection of localizing tensor ideals of $\ccat{K}$.

\subsection*{Radical tensor ideals}

A tensor ideal $\ccat{I} \subseteq\ccat{K}$ is \emph{radical} if
\begin{equation*}
X^{\otimes n} \in \ccat{I} \text{ for some positive integer } n \,\implies\, X \in \ccat{I}\,.
\end{equation*}
The \emph{radical closure} of a tensor ideal $\ccat{I} \subseteq\ccat{K}$ is the smallest radical tensor ideal containing $\ccat{I}$. This is well-defined, since the above property is preserved under intersection; that is the radical closure of $\ccat{I}$ is the intersection of all radical tensor ideals containing $\ccat{I}$.

Given $\ccat X\subseteq\ccat{K}$ we write $\rad[\alpha](\ccat X)$ for the smallest radical $\alpha$-localizing tensor ideal of $\ccat{K}$ containing $\ccat X$. A radical ideal of the form $\rad[\alpha](X)$ for some $X\in\ccat{K}$ is called \emph{principal}. We denote by $\Rad[\alpha](\ccat{K})$ the collection of radical $\alpha$-localizing tensor ideals of $\ccat{K}$.

When $\alpha = \aleph_0$, then the radical closure of an
$\alpha$-localizing tensor ideal $\ccat{I}$ is explicitly given by
\begin{equation*}
\rad[\alpha](\ccat{I})=\set{X \in \ccat{K}}{X^{\otimes n} \in \ccat{I} \text{ for some integer } n \geq 1}\,.
\end{equation*}
This need not hold for regular cardinals $\alpha > \aleph_0$; see \cite[Example~4.1.6]{Balchin/Stevenson:2021}.

When $\ccat{K}$ has arbitrary coproducts, then we set for $\ccat X\subseteq\ccat{K}$
\begin{equation*}
\rad(\ccat{X}) \colonequals \bigcup_{\alpha\text{ regular}} \rad[\alpha](\ccat{X})
\end{equation*}
and write $\Rad(\ccat{K})$ for the collection of radical localizing tensor ideals of $\ccat{K}$.

\begin{lemma}\label{le:Rad-set}
The following conditions are equivalent.
\begin{enumerate}
\item $\Rad[\alpha](\ccat{K})$ is a set.
\item $\set{\rad[\alpha](X)}{X \in \ccat{K}}$ is a set.
\end{enumerate}
In that case every ideal in $\Rad[\alpha](\ccat{K})$ is generated by a set of objects. The same holds when replacing $\Rad[\alpha]$ by $\Rad$ and $\rad[\alpha]$ by $\rad$, respectively.
\end{lemma}
\begin{proof}
If $\set{\rad[\alpha](X)}{X \in \ccat{K}}$ is a set, then for any $\ccat{I}\in\Rad[\alpha](\ccat{K})$ we have 
\begin{equation*}
\ccat{I}=\bigvee_{X\in\ccat{I}}\rad[\alpha](X)=\bigvee_{X\in\ccat I_0}\rad[\alpha](X)=\rad[\alpha](\ccat{I}_0)
\end{equation*}
for some subset $\ccat{I}_0\subseteq\ccat{I}$. Thus $\ccat{I}$ is generated by a set of objects and determined by a subset of $\set{\rad[\alpha](X)}{X \in \ccat{K}}$. It follows that $\Rad[\alpha](\ccat{K})$ is a set. The reverse direction is clear. The proof for $\Rad$ and $\rad$ works analogously.
\end{proof}

\begin{lemma} \label{RadAlphaFrame}
When $\Rad[\alpha](\ccat{K})$ is a set, then the radical $\alpha$-localizing tensor ideals $\Rad[\alpha](\ccat{K})$ form a frame. The same holds for $\Rad(\ccat{K})$.
\end{lemma}
\begin{proof}
It is enough to show that the infinite-join distributive law holds, that is
\begin{equation*}
\ccat{I} \wedge \bigvee A = \bigvee_{\ccat{J} \in A} (\ccat{I} \wedge \ccat{J}) \quad\text{for any } \ccat{I} \in \Rad[\alpha](\ccat{K}) \text{ and } A \subseteq \Rad[\alpha](\ccat{K})\,.
\end{equation*}
The inclusion $\supseteq$ is clear. For $\subseteq$, we take $X \in \ccat{I} \wedge \bigvee A$ and define
\begin{equation*}
\ccat{C}_X \colonequals \set{Y \in \bigvee A}{X \otimes Y \in \bigvee_{\ccat{J} \in A} (\ccat{I} \wedge \ccat{J})}\,.
\end{equation*}
It is straightforward to check that $\ccat{C}_X$ is a radical $\alpha$-localizing tensor ideal containing any $\ccat{J} \in A$. In particular, we have $\bigvee A \leq \ccat{C}_X$ and $X \in \ccat{C}_X$. The claim follows from the fact that $\bigvee_{\ccat{J} \in A} (\ccat{I} \wedge \ccat{J})$ is radical. 

The proof for $\Rad(\ccat{K})$ is the same.
\end{proof}

In the previous lemma the tensor product is essential. In general, distributivity fails in the lattice of $\alpha$-localizing subcategories.

\begin{example}
Consider the path algebra of the quiver $\circ\to\circ$ over a fixed field. We denote by $P_0$ and $P_1$ the indecomposable projective objects which are related by a monomorphism $P_1\to P_0$. The thick subcategories of the category of perfect complexes form a lattice which has the following Hasse diagram.
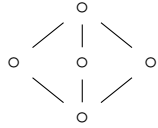
\begin{equation*}
\begin{tikzcd}[column sep = small,row sep=small]
&\circ\arrow[dash,ld]\arrow[dash,d]\arrow[dash,rd] \\
\circ\arrow[dash,rd] &\circ\arrow[dash,d]&\circ\arrow[dash,ld] \\
&\circ 
\end{tikzcd}
\end{equation*}
Thus the category has precisely three proper thick subcategories which are pairwise incomparable; they are generated by $P_0$, $P_1$, and the two term complex $P_1\to P_0$. Clearly, this lattice is not distributive.
\end{example}

\begin{lemma}\label{le:coproduct-join}
For a set $\ccat{X} \subseteq \ccat{K}$ with $|\ccat{X}| < \alpha$ we have
\begin{equation*}
\rad[\alpha](\coprod_{X \in \ccat{X}} X)=\bigvee_{X \in \ccat{X}}\rad[\alpha](X)\,.
\end{equation*}
\end{lemma}
\begin{proof}
The tensor ideal $\rad[\alpha](\coprod_{X \in \ccat{X}} X)$ is closed under direct summands. Thus $\ccat{X} \subseteq \rad[\alpha](\coprod_{X \in \ccat{X}} X)$ and the inclusion $\supseteq$ follows. On the other hand, $\coprod_{X \in \ccat{X}} X$ belongs to $\bigvee_{X \in \ccat{X}}\rad[\alpha](X)$ since $|\ccat{X}|<\alpha$, and this implies the inclusion $\subseteq$.
\end{proof}

\begin{lemma} \label{AlphaBuilding}
Let $\alpha$ be a regular cardinal, $\ccat{X} \subseteq \ccat{K}$ and $X \in \rad[\alpha](\ccat{X})$. Then there exists $\ccat{X}' \subseteq \ccat{X}$ with $|\ccat{X}'| < \alpha$ and $X \in \rad[\alpha](\ccat{X}')$.
\end{lemma}
\begin{proof}
It is enough to show that
\begin{equation*}
\rad[\alpha](\ccat{X}) = \bigcup_{\substack{\ccat{X}' \subseteq \ccat{X}\\|\ccat{X}'| < \alpha}} \rad[\alpha](\ccat{X}')\,.
\end{equation*}
The inclusion $\supseteq$ is obvious. For the reverse, it is enough to show the right-hand side is a radical $\alpha$-localizing triangulated subcategory of $\ccat{K}$. By construction, it is such. 
\end{proof}

\begin{lemma} \label{alphaCompactPrincipal}
The $\alpha$-compact elements of the frame $\Rad[\alpha](\ccat{K})$ are precisely the principal radical ideals of $\ccat{K}$.
\end{lemma}
\begin{proof}
Let $\ccat{I} \in \Rad[\alpha](\ccat{K})$ be an $\alpha$-compact element. We can write
\begin{equation*}
\ccat{I} = \bigvee_{X \in \ccat{I}} \rad[\alpha](X)\,.
\end{equation*}
By assumption $\ccat{I}$ is $\alpha$-compact, and there exists $\ccat{S} \subseteq \ccat{I}$, such that $|\ccat{S}| < \alpha$ and
\begin{equation*}
\ccat{I} = \bigvee_{X \in \ccat{S}} \rad[\alpha](X) = \rad[\alpha]\left(\coprod_{X \in \ccat{S}} X\right)
\end{equation*}
by \cref{le:coproduct-join}. 

For the converse direction, let $\ccat{I} = \rad[\alpha](X)$ be a principal ideal with
\begin{equation*}
\ccat{I} \leq \bigvee_{\ccat{J} \in A} \ccat{J}
\end{equation*}
for some $A \subseteq \Rad[\alpha](\ccat{K})$. Without loss of generality, we may assume every $\ccat{J}$ is a principal ideal, that is we assume there exists $\ccat{X} \subseteq \ccat{K}$ with $X \in \rad[\alpha](\ccat{X})$. By \cref{AlphaBuilding}, there exists $\ccat{X}' \subseteq \ccat{X}$ with $|\ccat{X}'| < \alpha$ and $X \in \rad[\alpha](\ccat{X}')$. Then every principal ideal is $\alpha$-compact.
\end{proof}

\subsection*{The tensor property}

Let $\ccat{K}$ be a tensor triangulated category. When $\ccat{K}$ has $\alpha$-coproducts we say $\ccat{K}$ satisfies the \emph{$\alpha$-tensor property} if for all $X,Y \in \ccat{K}$
\begin{equation*}
\rad[\alpha](X\otimes Y) = \rad[\alpha](X) \wedge \rad[\alpha](Y)\,.
\end{equation*}
When $\ccat{K}$ has arbitrary coproducts we drop $\alpha$ and say $\ccat{K}$ satisfies the \emph{tensor property} if for all $X,Y \in \ccat{K}$
\begin{equation*}
\rad(X\otimes Y) = \rad(X) \wedge \rad(Y)\,.
\end{equation*}

For $\alpha = \aleph_0$ the $\alpha$-tensor property always holds and
$\Rad[\alpha](\ccat{K})$ is a coherent frame; see
\cite{Balmer:2005,Kock/Pitsch:2017}. For general $\alpha$ we do not expect the $\alpha$-tensor
property to hold, though no examples seem to be known.

\subsection*{Support}

Let $\alpha$ be a regular cardinal and $\ccat{K}$ a tensor triangulated category with $\alpha$-coproducts. An \emph{$\alpha$-support} on $\ccat{K}$ is a pair $(F, \sigma)$ consisting of a frame $F$ and a map $\sigma\colon\Obj(\ccat{K})\to F$ satisfying:
\begin{enumerate}
\item[(S1)] $\sigma(0)=0$ and $\sigma(\one)=1$,
\item[(S2)] $\sigma(\Sigma X)=\sigma(X)$ for every $X \in \ccat{K}$,
\item[(S3)] $\sigma(\coprod_{X \in \ccat{X}} X)=\bigvee_{X \in \ccat{X}} \sigma(X)$ for every family $\ccat{X} \subseteq \ccat{K}$ with $|\ccat{X}|<\alpha$,
\item[(S4)] $\sigma(X\otimes Y)=\sigma(X)\wedge \sigma(Y)$ for every $X,Y\in\ccat{K}$, and
\item[(S5)] $\sigma(Y)\leq \sigma(X)\vee \sigma(Z)$ for every exact triangle $X\to Y\to Z\to
\Sigma X$ in
$\ccat{K}$.
\end{enumerate}
A morphism of supports from $(F, \sigma)$ to $(F', \sigma')$ is a frame map $\varphi\colon F\to F'$ satisfying $\sigma=\varphi\circ \sigma'$.

When $\ccat{K}$ has arbitrary coproducts, then we call a pair $(F,\sigma)$ a \emph{support} on $\ccat{K}$ if it is an $\alpha$-support for every cardinal $\alpha$.

The traditional notion of support for objects of a triangulated category uses a topological space. The following result demonstrates that the notion of a frame is the appropriate generalization; it is the analogue of \cite[Theorem~3.2.3]{Kock/Pitsch:2017}.

\begin{theorem}\label{th:Rad-support}
Let $\alpha$ be a regular cardinal and $\ccat{K}$ a tensor triangulated category with $\alpha$-coproducts. Suppose that $\ccat{K}$ satisfies the $\alpha$-tensor property and $\Rad[\alpha](\ccat{K})$ is a set. Then the frame $\Rad[\alpha](\ccat{K})$ is $\alpha$-coherent and the map
\begin{equation*}
\Obj(\ccat{K})\lto \Rad[\alpha](\ccat{K})\,,\quad X\mapsto \rad[\alpha](X)
\end{equation*}
is an $\alpha$-support; it is initial among all $\alpha$-supports on $\ccat{K}$. 
\end{theorem}
\begin{proof}
By \cref{alphaCompactPrincipal}, the $\alpha$-compact elements of $\Rad[\alpha]({\ccat{K}})$ are precisely the principal ideals. Since $\alpha$ is regular, they are closed under $\alpha$-joins, and by assumption, the principal ideals are also closed under finite meets. Lastly, every radical $\alpha$-ideal is the join of the principal ideals of its elements. Thus $\Rad[\alpha](\ccat{K})$ is $\alpha$-coherent.

We need to check for $\rad[\alpha](-)$ the defining conditions of an $\alpha$-support. Conditions (S1), (S2), and (S5) are clear. (S3) follows from \cref{le:coproduct-join}, and (S4) follows from the $\alpha$-tensor property. 

Now let $(F,\sigma)$ be an $\alpha$-support. We need to provide a unique frame morphism $\varphi\colon \Rad[\alpha](\ccat{K})\to F$. Since every element of $\Rad[\alpha](\ccat{K})$ is the join of principal ideals, by (S3) it is enough to define $\varphi$ on the principal ideals. We set
\begin{equation*}
\varphi(\rad[\alpha](X)) \colonequals \sigma(X)\,.
\end{equation*}
To show that this assignment is well-defined let $X,Y \in \ccat{K}$ with $\rad[\alpha](X)=\rad[\alpha](Y)$. Observe that $\set{Z \in \ccat{K}}{\sigma(Z)\leq\sigma(X)}$ is a radical $\alpha$-localizing tensor ideal containing $X$. Thus $\rad[\alpha](X)\supseteq\rad[\alpha](Y)$ implies $\sigma(X) \geq \sigma(Y)$. By symmetry the assertion follows. By construction, this morphism is unique. 
\end{proof}

The following result is the analogue of \cref{th:Rad-support}.

\begin{theorem}\label{th:Rad-support-infinite}
Let $\ccat{K}$ be a tensor triangulated category with arbitrary coproducts. Suppose that $\ccat{K}$ satisfies the tensor property and $\Rad(\ccat{K})$ is a set. Then the radical localizing tensor ideals of $\ccat{K}$ form a frame and the map
\begin{equation*}
\Obj(\ccat{K})\lto \Rad(\ccat{K})\,,\quad X\mapsto \rad(X)
\end{equation*}
is a support; it is initial among all supports on $\ccat{K}$. 
\end{theorem}
\begin{proof}
Adapt the proof of \cref{th:Rad-support}.
\end{proof}

\subsection*{Prime tensor ideals}

An $\alpha$-localizing tensor ideal $\ccat P\subseteq\ccat{K}$ is called \emph{prime}, if
\begin{equation*}
X \otimes Y \in \ccat{P} \text{ for } X, Y \in \ccat{K} \,\implies\, X \in \ccat{P} \text{ or } Y \in \ccat{P}\,.
\end{equation*}
Any prime tensor ideal is radical. 

\begin{lemma} \label{PrimeElementsIdeals}
Suppose that $\ccat{K}$ satisfies the $\alpha$-tensor property and $\Rad[\alpha](\ccat{K})$ is a set. The prime elements of $\Rad[\alpha](\ccat{K})$ are precisely the prime $\alpha$-localizing tensor ideals of $\ccat{K}$.
\end{lemma}
\begin{proof}
Let $\ccat{P}$ be a prime $\alpha$-localizing tensor ideal of $\ccat{K}$. Assume $\ccat{I} \wedge \ccat{J} \leq \ccat{P}$ for some $\ccat{I}, \ccat{J} \in \Rad[\alpha](\ccat{K})$. That is the intersection $\ccat{I} \cap \ccat{J}$ is contained in $\ccat{P}$. We assume $\ccat{I} \not\subseteq \ccat{P}$. Then there exists $X \in \ccat{I}$ with $X \notin \ccat{P}$, and we obtain
\begin{equation*}
\left(\forall Y \in \ccat{J}: X \otimes Y \in \ccat{I} \cap \ccat{J} \subseteq \ccat{P} \,\implies\, Y \in \ccat{P}\right) \,\implies\, \ccat{J} \subseteq \ccat{P}\,.
\end{equation*}
That is $\ccat{P}$ is a prime element of the frame $\Rad[\alpha](\ccat{K})$. 

For the reverse direction, we assume $\ccat{P}$ is a prime element of $\Rad[\alpha](\ccat{K})$. Assume $X \otimes Y \in \ccat{P}$. Then by the $\alpha$-tensor property
\begin{equation*}
\rad[\alpha](X) \wedge \rad[\alpha](Y) = \rad[\alpha](X \otimes Y) \subseteq \ccat{P}\,,
\end{equation*}
and by assumption $\rad[\alpha](X) \subseteq \ccat{P}$ or $\rad[\alpha](Y) \subseteq \ccat{P}$. That is $X \in \ccat{P}$ or $Y \in \ccat{P}$, and $\ccat{P}$ is a prime $\alpha$-localizing tensor ideal.
\end{proof}

\subsection*{The spectrum}

Let $\alpha$ be a regular cardinal and $\ccat{K}$ a tensor triangulated category with $\alpha$-coproducts. Suppose that $\ccat{K}$ satisfies the $\alpha$-tensor property and that $\Rad[\alpha](\ccat{K})$ is a set. Then we define its \emph{$\alpha$-spectrum} to be the space associated to the frame $\Rad[\alpha](\ccat{K})$, so
\begin{equation*}
\Spec^\alpha(\ccat{K})\colonequals\pt (\Rad[\alpha](\ccat{K}))\,.
\end{equation*}
Analogously, we set
\begin{equation*}
\Spec(\ccat{K})\colonequals\pt(\Rad(\ccat{K}))
\end{equation*}
when $\ccat{K}$ has arbitrary coproducts, assuming $\ccat{K}$ satisfies the tensor property and $\Rad(\ccat{K})$ is a set. We hasten to add that these spaces may be empty, except when the corresponding frames are spatial, e.g.\@ when $\alpha = \aleph_0$. For $\alpha > \aleph_0$ interesting examples arise by taking the categories of $\alpha$-compact objects of the compactly generated tensor triangulated categories which are listed in the introduction.

%%%%%%%%%%%%%%%%%%%%%%%%%%%%%%%%%%%%%%%%%%%%%%%%%
%%%%%%%%%%%%%%%%%%%%%%%%%%%%%%%%%%%%%%%%%%%%%%%%%
\section{Well-generated tensor triangulated categories}

In this section we study well-generated triangulated categories as introduced by Neeman \cite{Neeman:2001}. This includes the class of compactly generated categories. Each well-generated triangulated category admits a filtration given by the full subcategories of $\alpha$-compact objects, where $\alpha$ runs through all regular cardinals. This filtration induces a filtration of the frame of radical localizing tensor ideals when the category is tensor triangulated.

A motivation for studying well-generated triangulated categories is the fact that localizing subcategories or Verdier quotient categories of a compactly generated category are usually not compactly generated. However, the class of well-generated categories is closed under localizing subcategories generated by sets of objects and under Verdier quotients with respect to such localizing subcategories. Thus the class of well-generated tensor triangulated categories seems to be the appropriate universe for studying tensor triangulated categories with arbitrary coproducts.

\subsection*{Well-generated triangulated categories}

Let $\ccat{T}$ be a triangulated category with arbitrary coproducts. For a regular cardinal $\alpha$, an object $X\in\ccat{T}$ is \emph{$\alpha$-small} if every morphism $X\to\coprod_{X' \in \ccat{X}} X'$ factors through $\coprod_{X' \in \ccat{X}'} X'$ for some subset $\ccat{X}' \subseteq \ccat{X}$ with $|\ccat{X}'|<\alpha$. We denote by $\compact[\alpha]{\ccat{T}}$ the full subcategory of \emph{$\alpha$-compact objects} in $\ccat{T}$. This is the maximal $\alpha$-perfect class of $\alpha$-small objects of $\ccat{T}$, and it is an $\alpha$-localizing subcategory of $\ccat{T}$. When $\alpha = \aleph_0$, then the notions $\aleph_0$-small and $\aleph_0$-compact coincide and such objects are called \emph{compact}.

A triangulated category $\ccat{T}$ with arbitrary coproducts is \emph{well-generated} if there exists a regular cardinal $\alpha$ and a set of $\alpha$-compact objects generating $\ccat{T}$; see \cite{Neeman:2001,Krause:2001} for details. In that case one calls $\ccat{T}$ \emph{$\alpha$-compactly generated}. If $\alpha=\aleph_0$, then $\ccat{T}$ is called \emph{compactly generated}.

Now suppose that $\ccat{T}$ is $\alpha$-compactly generated, and therefore $\beta$-compactly generated for every regular cardinal $\beta \geq \alpha$. Then
\begin{equation*}
\compact[\beta]{\ccat{T}} = \loc[\beta](\compact[\alpha]{\ccat{T}})\,.
\end{equation*}
Moreover, by \cite[Proposition~8.4.2]{Neeman:2001} the category $\compact[\alpha]{\ccat{T}}$ is essentially small and 
\begin{equation*}
\ccat{T} = \bigcup_{\beta\geq \alpha} \compact[\beta]{\ccat{T}}\,,
\end{equation*}
where $\beta$ runs through all regular cardinals. In particular, every $\alpha$-localizing subcategory of $\compact[\alpha]{\ccat{T}}$ is generated by a set and the $\alpha$-localizing subcategories of $\compact[\alpha]{\ccat{T}}$ form a set, and thus a lattice.

\subsection*{Extension and restriction}

Fix a well-generated tensor triangulated category $(\ccat{T},\otimes,\one)$. We assume the localizing tensor ideals of $\ccat{T}$ that are generated by sets of objects form a set. In that case every localizing tensor ideal is actually generated by a set of objects, cf.\@ \cref{le:Rad-set}.

Let $\alpha \leq \beta$ be regular cardinals. Then there is a pair of maps
\begin{equation*}
\begin{tikzcd}[column sep=huge]
\Loc[\alpha](\ccat{T^\alpha})\ar[r,shift left] & \Loc[\beta](\ccat{T^\beta}) \ar[l,shift
left]
\end{tikzcd}
\end{equation*}
given by
\begin{equation*}
\ccat{I}\mapsto \loc[\beta](\ccat{I})\qquad\text{and} \qquad\ccat{J}
\cap \compact[\alpha]{\ccat{T}}\mapsfrom\ccat J\,.
\end{equation*}
The first map is \emph{extension to $\compact[\beta]{\ccat{T}}$} and the second \emph{restriction to $\compact[\alpha]{\ccat{T}}$}. Both maps preserve the order, but in general neither is a morphism of frames: Extension preservers joins but need not preserve meets and restriction preserves meets but need not preserve joins. 

From the description of the maps above we obtain
\begin{equation} \label{eq:ExtRestContainment}
\ccat{I}\leq \loc[\beta](\ccat{I}) \cap \compact[\alpha]{\ccat{T}} \qquad\text{and}\qquad \loc[\beta](\ccat{J} \cap \compact[\alpha]{\ccat{T}}) \leq \ccat{J}\,.
\end{equation}
When $\ccat{T}$ is $\alpha$-compactly generated, then the first inclusion becomes equality. This follows from Thomason's localization theorem, which was first proved by Thomason \cite{Thomason/Trobaugh:1990}. In this generality it is due to Neeman \cite[Theorem~1.14]{Neeman:2001}. Furthermore the same holds when we replace $\Loc[\beta](\compact[\beta]{\ccat{T}})$ by $\Loc(\ccat{T})$, and $\loc[\beta](-)$ by $\loc(-)$. If we think of the lattices underlying the frames as categories as in the discussion before \cref{CohcatDistcatEquiv}, then extension is left adjoint to restriction. In particular, \eqref{eq:ExtRestContainment} provides unit and counit of the adjunction.

The extension and restriction maps induce maps $\extend{\alpha}{\beta}{(-)}$ and $\restrict{\beta}{\alpha}{(-)}$, respectively, on the frame of radical $\alpha$-localizing tensor ideals by taking radical closures. Restriction preserves the radical property and so the maps are given by:
\begin{equation*}
\extend{\alpha}{\beta}{\ccat{I}} \colonequals \rad[\beta](\ccat{I}) \quad\text{and}\quad \restrict{\beta}{\alpha}{\ccat{J}} \colonequals \ccat{J} \cap \compact[\alpha]{\ccat{T}}
\end{equation*}
for $\ccat{I} \in \RadT{\alpha}$ and $\ccat{J} \in \RadT{\beta}$. Again, these need not be morphisms of frames. 

We obtain the same adjointness relations as in \eqref{eq:ExtRestContainment}; when $\ccat{T}$ is $\alpha$-compactly generated, then:
\begin{equation} \label{eq:ExtRestWellGenerated}
\restrict{\beta}{\alpha}{(\extend{\alpha}{\beta}{\ccat{I}})} = \ccat{I} \quad\text{and}\quad \extend{\alpha}{\beta}{(\restrict{\beta}{\alpha}{\ccat{J}})} \leq \ccat{J}\,.
\end{equation}
In particular, this means that extension is an injective map.

\begin{lemma} \label{ExtensionInjective}
Let $\alpha \leq \beta$ be regular cardinals and $\ccat{T}$ $\alpha$-compactly generated. We assume
\begin{equation*}
\rad[\beta](X\otimes Y) = \rad[\beta](X) \wedge \rad[\beta](Y)
\end{equation*}
for any $X,Y \in \compact[\alpha]{\ccat{T}}$. Then
\begin{enumerate}
\item $\compact[\alpha]{\ccat{T}}$ satisfies the $\alpha$-tensor property, and
\item the extension map $\extend{\alpha}{\beta}{(-)}$ is a morphism of frames.
\end{enumerate}
\end{lemma}
\begin{proof}
The first claim follows from the fact that the restriction map preserves meets. 

Since every element in $\Rad[\alpha](\compact[\alpha]{\ccat{T}})$ can be written as the join of principal elements and extension preserves joins, it is enough to show that extension preserves the meet of principal elements. But this is clear, since the meet of principal elements is given by the tensor products. 
\end{proof}

\begin{lemma} \label{RadalphaSpatial}
Let $\alpha \leq \beta$ be regular cardinals and $\ccat{T}$ $\alpha$-compactly generated. We assume
\begin{equation*}
\rad[\beta](X\otimes Y) = \rad[\beta](X) \wedge \rad[\beta](Y)
\end{equation*}
for any $X,Y \in \compact[\alpha]{\ccat{T}}$. If the frame $\RadT{\beta}$ is spatial, then $\RadT{\alpha}$ is spatial.
\end{lemma}
\begin{proof}
By \cref{ExtensionInjective}, the extension map $\extend{\alpha}{\beta}{(-)}$ is an inclusion of frames. For any pair of ideals $\ccat{I}\not\leq\ccat J$ in $\RadT{\alpha}$ there is point $x\colon \RadT{\beta}\to\{0,1\}$ separating $\extend{\alpha}{\beta}{\ccat{I}}$ from $\extend{\alpha}{\beta}{\ccat{J}}$. Then the composite $x\circ \extend{\alpha}{\beta}{(-)}$ separates $\ccat{I}$ from $\ccat J$. Thus $\RadT{\alpha}$ has enough points.
\end{proof}

\subsection*{Stratification}

For any well-generated tensor triangulated category we introduce a notion of stratification, following the concepts from \cite{Benson/Iyengar/Krause:2011,Barthel/Heard/Sanders:2021} for compactly generated triangulated categories.

\begin{definition} \label{dfn:compactly-stratified}
Let $\ccat{T}$ be a well-generated tensor triangulated category, and assume the localizing tensor ideals of $\ccat{T}$ form a set. We call $\ccat{T}$ \emph{$\alpha$-compactly stratified} if
\begin{enumerate}
\item[(CS1)] the triangulated category $\ccat{T}$ is $\alpha$-compactly generated,
\item[(CS2)] the frame $\Rad(\ccat{T})$ is spatial, 
\item[(CS3)] the tensor triangulated category $\ccat{T}$ satisfies the tensor property,
\item[(CS4)] the induced map $\pt(\Rad(\ccat{T})) \to \pt(\Rad[\alpha](\compact[\alpha]{\ccat{T}}))$ is a bijection, and
\item[(CS5)] every point of $\Rad[\alpha](\compact[\alpha]{\ccat{T}})$ is locally closed. 
\end{enumerate}
As before we drop the cardinal $\alpha=\aleph_0$ and call $\ccat{T}$ \emph{compactly stratified} if it is $\aleph_0$-compactly stratified.
\end{definition}

The conditions (CS1)--(CS3) imply that the frame $\Rad[\alpha](\compact[\alpha]{\ccat{T}})$ is spatial and that extension yields an injective morphism $\Rad[\alpha](\compact[\alpha]{\ccat{T}})\to\Rad(\ccat{T})$ of frames. This follows from \cref{ExtensionInjective,RadalphaSpatial}. In particular, the map in (CS4) is well-defined.

\begin{theorem} \label{bijectionPointsRad}
Let $\ccat{T}$ be a well-generated tensor triangulated category and let $\alpha \leq \beta$ be regular cardinals. If $\ccat{T}$ is $\alpha$-compactly stratified, then $\ccat{T}$ is $\beta$-compactly stratified.
\end{theorem}
\begin{proof}
By \cref{RadalphaSpatial}, the frames $\RadT{\alpha}$ and $\RadT{\beta}$ are spatial. It follows from \cref{ExtensionInjective} that the extension maps $\extend{\alpha}{\beta}{(-)}$ and $\extend{\beta}{{}}{(-)}$ are injective morphisms of spatial frames. Then \cref{BijPoints} implies that the induced map $\pt(\Rad(\ccat{T})) \to \pt (\RadT{\beta})$ is a bijection. That every point of $\RadT{\beta}$ is locally closed follows from \cref{strongly-visible}.
\end{proof}

\cref{th:cofiltration} from the introduction is a consequence. We use the fact that any set of objects of a well-generated category is $\alpha$-compact for some sufficiently big regular cardinal $\alpha$.

\begin{proof}[Proof of \cref{th:cofiltration}]
We assume that $\ccat{T}$ is $\alpha$-compactly stratified for some
regular cardinal $\alpha$. Then for each regular cardinal $\beta\ge\alpha$ the inclusion $\compact[\beta]{\ccat{T}}\to\ccat{T}$ induces a frame morphism $\RadT{\beta} \to \Rad(\ccat{T})$. This yields a compatible family of continuous maps $\Spec(\ccat{T})\to \Spec^\beta(\compact[\beta]{\ccat{T}})$, which are bijections (thanks to \cref{bijectionPointsRad}) and become homeomorphisms when all ideals in $\Rad(\ccat{T})$ are generated by $\beta$-compact objects.
\end{proof}

The following corollary produces further classes of $\alpha$-compactly stratified categories.

\begin{corollary}\label{co:bijectionPointsRad} 
Let $\ccat{T}$ be a well-generated tensor triangulated category and suppose that $\ccat{T}$ is $\alpha$-compactly stratified. If $\ccat{S}\subseteq\ccat{T}$ is a radical localizing tensor ideal, then the quotient $\ccat{T}/\ccat{S}$ is $\beta$-compactly stratified for some regular cardinal $\beta\geq \alpha$.
\end{corollary}

\begin{proof}
First observe that any radical localizing tensor ideal of $\ccat{T}$ is generated by a set of objects, because we assume that $\Rad(\ccat{T})$ is a set, cf.\@ \cref{le:Rad-set}. Thus we find $\beta\ge\alpha$ such that $\ccat{S}$ is generated by a set of $\beta$-compact objects. The category $\ccat{T}$ is $\beta$-compactly stratified by \cref{bijectionPointsRad}. Then the quotient $\ccat{T}/\ccat{S}$ inherits from $\ccat{T}$ the structure of a $\beta$-compactly generated tensor triangulated category, cf.\@ \cite[Theorem~1.14]{Neeman:2001}. The assignment $\ccat{I}\mapsto \ccat{I}\vee \ccat{S}$ gives a surjective frame morphism $\Rad(\ccat{T})\to \Rad(\ccat{T}/\ccat{S})$, because the quotient functor $\ccat{T}\to\ccat{T}/\ccat{S}$ identifies the localizing tensor ideals of $\ccat{T}$ containing $\ccat{S}$ with the localizing tensor ideals of $\ccat{T}/\ccat{S}$. It follows that $\Rad(\ccat{T}/\ccat{S})$ is a spatial frame satisfying the tensor property. The induced map
\begin{equation*}
\pt(\Rad(\ccat{T}/\ccat{S}))\lto \pt(\Rad(\ccat{T}))
\end{equation*}
is injective and identifies $\pt(\Rad(\ccat{T}/\ccat{S}))$ with the closed set $\pt(\Rad(\ccat{T}))\setminus U(\ccat{S})$. We have the analogous properties for the quotient functor 
\begin{equation*}
\compact[\beta]{\ccat{T}} \lto \compact[\beta]{\ccat{T}}/\compact[\beta]{\ccat{S}} \longiso \compact[\beta]{(\ccat{T}/\ccat{S})}\,,
\end{equation*}
where the last functor is an equivalence except when $\beta = \aleph_0$, in which case it is an equivalence up to direct summands. In particular, the induced map
\begin{equation*}
\pt(\Rad[\beta](\compact[\beta]{(\ccat{T}/\ccat{S})})) \lto \pt(\RadT{\beta})
\end{equation*}
identifies $\pt(\Rad[\beta](\compact[\beta]{(\ccat{T}/\ccat{S})}))$ with $\pt(\RadT{\beta})\setminus U(\compact[\beta]{\ccat{S}})$. The bijection 
\begin{equation*}
\pt(\Rad(\ccat{T})) \longiso \pt(\RadT{\beta})
\end{equation*}
identifies $U(\ccat{S})$ with $U(\compact[\beta]{\ccat{S}})$, since $\rad(\compact[\beta]{\ccat{S}})=\ccat{S}$, and this yields the bijection
\begin{equation*}
\pt(\Rad(\ccat{T}/\ccat{S})) \longiso \pt(\Rad[\beta](\compact[\beta]{(\ccat{T}/\ccat{S})}))\,.
\end{equation*}
That every point of $\Rad[\beta](\compact[\beta]{(\ccat{T}/\ccat{S})})$ is locally closed follows from the corresponding property of $\RadT{\beta}$ with \cref{strongly-visible}.
\end{proof}

\begin{remark}
(1) In \cref{dfn:compactly-stratified}, when $\alpha = \aleph_0$, the condition that every point is locally closed is satisfied when the Balmer spectrum (i.e.\@ the space of points of the Hochster dual of $\Rad[\alpha](\compact[\alpha]{\ccat{T}})$) is a noetherian space. This follows from \cref{pointsClosedDownwardsOpen}.

(2) Let $\ccat{T}$ be compactly stratified and suppose that all compact objects are rigid. Then $\Spec(\ccat{T})$ is a discrete space. To see this, fix $x\in\Spec(\ccat{T})$ and write $\{x\}=U\cap V$ as the intersection of an open subset $U$ and a closed subset $V$, which are pre-images of an open subset $U'$ and a closed subset $V'$ under the canonical map $\Spec(\ccat{T})\to \Spec^{\aleph_0}(\ccat{T}^{\aleph_0})$. The open complement of $V'$ corresponds to a thick tensor ideal $\ccat{I}'$ of $\ccat{T}^{\aleph_0}$. Extension $\Rad[\aleph_0](\ccat{T}^{\aleph_0})\to \Rad(\ccat{T})$ maps $\ccat{I}'$ to the localizing tensor ideal $\ccat{I}=\rad(\ccat{I}')$ of $\ccat{T}$ that corresponds to the open complement of $V$ in $\Spec(\ccat{T})$. Then $\ccat{I}$ admits the perpendicular category $\ccat{I}^\perp = \set{Y\in\ccat{T}}{\Hom(X,Y)=0\text{ for all }X\in\ccat{I}}$ as a complement in $\Rad(\ccat{T})$. That $\ccat{I}^\perp$ is a localizing tensor ideal uses the fact that $\ccat{I}$ is generated by rigid objects; see \cite[Theorem~3.3.5]{Hovey/Palmieri/Strickland:1997}. Because a complement in a distributive lattice is unique, it follows that $V$ is open, and therefore $\{x\}$ is open.

(3) The examples of compactly generated tensor triangulated categories which are listed in the introduction are compactly stratified. When $\ccat{T}= \cat{D}(A)$ equals the derived category of a commutative noetherian ring $A$, there are explicit (though not optimal) bounds for $\alpha$ such that $\Spec^\alpha(\compact[\alpha]{\ccat{T}})\cong\Spec(\ccat{T})$ is discrete; see \cite[Proposition~6.5.4]{Balchin/Stevenson:2021}. Not much seems to be known for the derived category $\cat{D}(A)$ when $A$ is not noetherian, though there is an example showing that $\cat{D}(A)$ need not be compactly stratified \cite{Neeman:2000}.

(4) Let $\ccat{T}$ be a compactly generated tensor triangulated category such that compact and rigid objects coincide. In this context Barthel, Heard, and Sanders propose in \cite{Barthel/Heard/Sanders:2021} a definition of `stratification', using the notion of support from \cite{Balmer/Favi:2011}. Fortunately, this definition matches the one from \cref{dfn:compactly-stratified} for $\alpha=\aleph_0$.
\end{remark}

% \bibliographystyle{amsalpha}
% \bibliography{references}

\end{document}